\documentclass[11pt]{article}
\usepackage{amsmath,amssymb,amsfonts,amsthm,mathtools}
\usepackage{mathrsfs}
\usepackage{stackengine}
\usepackage{empheq}
\usepackage[T1]{fontenc}
\usepackage{lmodern}
\usepackage{graphicx}
\usepackage[margin=2.5cm]{geometry}
\usepackage{setspace}
\usepackage{titlesec}
\usepackage{titling}
\usepackage{enumitem}
\usepackage{romannum}
\usepackage{color}
\usepackage[english]{babel}
\usepackage{csquotes}
\usepackage[numbers,sort&compress]{natbib}
\usepackage{hyperref}

\newtheorem{theorem}{Theorem}[section]

\newtheorem{corollary}{Corollary}[theorem]
\newtheorem{lemma}[theorem]{Lemma}
\theoremstyle{definition}
\newtheorem{definition}{Definition}[section]
\newtheorem{prop}[theorem]{Proposition}

\theoremstyle{remark}
\newtheorem{remark}[theorem]{Remark}

\newcommand{\subjclass}[1]{\small	\par\noindent\textbf{Mathematics Subject Classification (2020): }#1\par}
\newcommand{\keywords}[1]{\small	\par\noindent\textbf{Keywords: }#1\par}
\definecolor{deepblue}{rgb}{0.0,0.0,0.7}

\definecolor{green}{rgb}{0.1,0.62,0.0}
\definecolor{orange}{rgb}{0.9,0.3,0.0}
\definecolor{blue}{rgb}{0.0, 0.0, 1.0}
\definecolor{red}{rgb}{1.0, 0.0, 0.0}

\font\myfont=cmbx12 at 11pt
\font\myfond=cmr12 at 10pt
\title{\myfont SEPARATION PROPERTY FOR THE NONLOCAL CAHN-HILLIARD-BRINKMAN SYSTEM WITH  SINGULAR POTENTIAL AND DEGENERATE MOBILITY}
\author{\myfond SHEETAL DHARMATTI* \thanks{* Corresponding author, \textbf{Correspondence to:} { sheetal@iisertvm.ac.in \; } \\ School of Mathematics, Indian Institute of Science Education and Research, Thiruvananthapuram, Maruthamala PO, Vithura, Thiruvananthapuram, Kerala, 695551, INDIA}  \;  and  \;  GREESHMA K \thanks{ greeshmak21@iisertvm.ac.in \; School of Mathematics, Indian Institute of Science Education and Research, Thiruvananthapuram, Maruthamala PO, Vithura, Thiruvananthapuram, Kerala, 695551, INDIA 
\\ \textbf{Acknowledgements :} Greeshma K would like to thank the Department of Science and Technology (DST), India, for the Innovation in Science Pursuit for Inspired Research (INSPIRE) Fellowship (IF210199). The work of Sheetal Dharmatti is supported by SERB  grant SERB-CRG/2021/008278 of the Government of India.
	}		
}
\date{}
\numberwithin{equation}{section}
\thanksmarkseries{arabic}
\AtBeginDocument{\pagenumbering{arabic}}
\begin{document}
\maketitle
\begin{abstract}
 This work studies the nonlocal Cahn-Hilliard-Brinkman (CHB) system, which models the evolution of a binary fluid in porous media. We focus on a system with a singular potential and a degenerate mobility function. The singular potential introduces challenges due to the blow-up of its derivatives near pure phases, while the degenerate mobility complicates the analysis. Our main result is the instantaneous separation property; that is, at any time, the solution stays separated from the pure phases. We adopt a new method, inspired by De Giorgi’s iteration, which was introduced for the two-dimensional Cahn-Hilliard (CH) equation with constant mobility. We prove that the separation scale can be estimated explicitly, and by treating the external force applied in the velocity equation as a control, the separation can be determined. Further, we discuss the long time behaviour of a weak solution in terms of its convergence to an equilibrium solution.
 
\end{abstract}

\setcounter{section}{1}

\section*{Introduction}
The Cahn-Hilliard-Brinkman (CHB) system describes the dynamics of two immiscible, incompressible and isothermal fluids within a bounded domain and porous media. This system consists of the Cahn-Hilliard (CH) equation, which governs the relative concentration of the fluids and the Brinkman equation, a modified Darcy's law that describes the fluid velocity. Let $\Omega$ be an open, bounded subset of $\mathbb{R}^2$. The average velocity of binary fluid is denoted by $\textbf{u}$ and the relative concentration of fluids is defined as $\varphi =(\varphi_1-\varphi_2)$, where $\varphi_1$, $\varphi_2$ represents the concentration of each fluid. The variable $\varphi$ takes values in the interval $[-1,1]$ with $\varphi =\pm 1$ corresponding to the pure phases. The nonlocal CHB system is given by,
\begin{align}
\varphi'+\nabla\cdot(\ \textbf{u}\varphi )\ &=\ \nabla\cdot( m(\varphi)\nabla\mu),     \; \text{in}\;\; \Omega\times[0,T],\label{eq1}
\\\mu &=\ a\varphi- J \ast\varphi + F '\left(\varphi\right),    \; \text{in}\;\; \Omega\times[0,T],\label{eq2}
\\-\nabla\cdot\left(\nu\left(\varphi\right)\nabla \textbf{u} \right) + \eta \textbf{u} + \nabla \pi &=\ \mu\nabla\varphi + \textbf{h},   \; \text{in}\;\; \Omega\times[0,T],\label{eq3}
\\ \nabla\cdot \textbf{u}  & =  0, \, \text{in}  \;\; \Omega\times[0,T],\label{eq4}
\\ \frac{\partial \varphi}{\partial n} = \frac{\partial \mu}{\partial n} &= 0 ,\;\; \text{on} \; \partial\Omega\times[0,T], \label{eq05}
\\\textbf{u} &= \textbf{0},\; \; \text{on} \;   \partial\Omega\times[0,T],\label{eq06}
\\ \varphi\left(. , 0 \right) &=\varphi_{0},\; \; \text{in} \; \; \Omega.\label{eq07}
\end{align}
Here, $\mu$ denotes the chemical potential, and the system is nonlocal due to the convolution term $J\ast \varphi$,  where $J$ is the interaction kernel. The functions $m$, $\nu$, and $\eta$ represent the fluid's mobility, viscosity and permeability, respectively. The term $F$ is the potential, and $\textbf{h}$ is an external forcing term and $a(x) =\int_{\Omega}J(x-y)dy$. The system enforces incompressibility and a no-slip boundary condition for the velocity field $\textbf{u}$ and a no-flux condition for the chemical potential, $\mu$.

In most practical problems, the potential appears to be singular, and such singularities introduce significant analytical difficulties, as derivatives of the potential blow up near the phase boundaries. A primary example of singular potential is the logarithmic potential, 
\begin{equation}\label{eq09}
F_{\log}(r)=((1+r) \log(1+r)+(1-r) \log(1-r)),\quad  r \in(-1,1).
\end{equation}
where $0<\theta <\theta_c$. Moreover, in many practical settings, the mobility function is degenerate, vanishing at pure phases, for instance,
\begin{equation}\label{eq1.0}
    m(r) = (1-r^2)^{1/p}, \quad  r \in(-1,1),
\end{equation}
where $p>1$. The combination of singular potential and degenerate mobility poses a significant analytical challenge as obtaining uniform bounds for the order parameter becomes nontrivial. Previous studies on the CHB system have mainly focused on the existence, uniqueness and regularity of a solution \cite{MCA, DPG, SDK, SDM}. In particular, \cite{CFG} proved the existence of weak solutions for the system \eqref{eq1}-\eqref{eq07} with constant mobility and a singular potential, while \cite{SDK} addressed the existence and uniqueness of strong solutions under higher regularity assumptions on the initial data and analysed a tracking-type optimal control problem.

A fundamental qualitative property of CH-type equations is the separation property. This property asserts that, after some finite time, the solution stays uniformly separated from the pure phases or in other words, the phase separation does not occur. For the CH equation, the separation property has been extensively studied under various settings, including constant mobility, and non-degenerate mobility \cite{GGG, AGS, POI}. In \cite{CGGG} the separation as well as asymptotic behaviour of solution of the 2- dimensional local Cahn Hilliard system with a non degenerate mobility are investigated. The separation property is derived using uniform in time estimates for the solution, together with the singularity of the potential near the pure phases. However, the case involving nonlocal interactions, singular potentials, and degenerate mobilities has been explored only to a limited extent.

The separation property of the Cahn-Hilliard (CH) equation is first studied in \cite{GGG} for the case of singular potential and constant mobility using Alikakos-Moser iterative arguments. Additionally, the long time behaviour of the solution is studied for dimension 2. Similar results for the nonlocal Cahn-Hilliard Hele-Shaw (CHHS) system with logarithmic potential are established in \cite{DGG}.

In  \cite{GGA} an approach based on the De-Giorgi iteration method is introduced to prove the separation property for the two-dimensional CH equation with constant mobility and streamlined the proof for both local and nonlocal potentials. This method allows a comparatively broader class of potentials provided that they satisfy a blow-up condition of $F''$. This method enables us to explicitly evaluate the separation scales, indicating how far the state is from the pure phases. However, this method relies solely on a uniform global in time estimate for the $L^p$ norm of $F''(\varphi)$. The analysis in \cite{POI, AGS} extends similar results for the three-dimensional system. In \cite{POI}, the author derived a Poincaré inequality for the time dependent truncated functions $\varphi_n$ to avoid time integrals of $\Bar{\varphi}_n$. Further, the observation $|\varphi_n| \leq 2\delta$ improved the estimate of the separation scale in \cite{POI}. Building on the observation in \cite{POI};  \cite{AGS} provided a simplified proof of the separation property for the three-dimensional Cahn-Hilliard equation without relying on this version of the Poincaré inequality and also investigated the existence of a global attractor for the system. In \cite{GAP}, a novel argument based on De Giorgi iteration is used to study the long time behaviour of solution of the three-dimensional Cahn-Hilliard equation with a non-degenerate mobility, without directly proving the separation property.
Recently, \cite{LUW25, LUWU, LUH, STG} used the De-Giorgi iterative method to prove the separation property for Cahn-Hilliard system involving dynamic boundary conditions, all these works assumed the mobility is either constant or non-degenerate.  

 The phase field equation, a variant of the CH equation, has been studied in \cite{LSP, LPH}. In \cite{LSP}, the convergence of the solution to a single equilibrium is established by exploiting a generalised \L{}ojasiewicz- Simon inequality. They proved the separation property for the potential,
\begin{equation}\label{eq08}
F(r)= r\log r + (1-r)\log(1-r), \quad  r \in(0,1).
\end{equation}
and mobility, 
\begin{equation}\label{eq1.11}
m(r)= \frac{1}{r(1-r)}, \quad  r \in(0,1).
\end{equation}
for which $mF''$ is a constant. The proof employs the Alikakos-Moser iterative method. Similar results under comparatively weaker assumptions on the potential were provided in \cite{LPH} with a simplified proof. In  \cite{LPH} authors assume $mF'' \geq \alpha_0$ for a constant $\alpha_0 >0$ and $\frac{1}{F''}$,  is strictly concave. Whereas in  \cite{GGF} resuls are proved for the nonlocal CH equation  for a potential and mobility satisfying slightly relaxed assumptions as compared to \cite{LSP}, specifically, $\exists M, \gamma >0$ such that $M(s) \leq M(1-s^2)$ and $m(s)F''(s) \geq \gamma,  \forall s\in [-1,1]$. In \cite{SFG}, the authors conclude that even for the Cahn-Hilliard Navier-Stokes (CHNS) system with a degenerate mobility and singular potential, the separation property follows as in \cite{GGF} since the proof completely relies on the CH equation. In a recent work \cite{FRG26}, the strict and instantaneous separation property of nonlocal Cahn-Hilliard system with a degenerate mobility, singular potential and source terms is studied using the same iterative approach. Using this method separation property results can be obtained for CHB system as well for the potential \eqref{eq08} and mobility \eqref{eq1.11}.

In this work our aim is to apply De-Giorgi's iterative method to the coupled CHB system with a singular potential and degenerate mobility to prove the separation property. We adopt the framework introduced in \cite{GGA} to prove the result. The primary difficulty in implementing the method from \cite{GGA} is in obtaining a global estimate for the $L^p$ norm of $\nabla F(\varphi)$ or $\nabla F'(\varphi)$ uniform in time. More precisely, the degenerate mobility restrain us from obtaining an estimate for $\nabla \mu$. In spite of the complexity, a modified assumption on the degenerate mobility and singular  potential (see \textbf{[A3]}) provides a significant analytical advantage. This enables us to derive global in time estimates for the system with a general potential and mobility which satisfies assumptions $\textbf{[A1]}$- $\textbf{[A4]}$ detailed in the next section. Thus our approach employs a less restrictive assumptions on the mobility and potential to achieve the separation as well as an explicit estimate for the separation scale. Further, the separation estimate involves a constant depending on the norm of the external source term in velocity equation, thereby obtaining the contribution of the Brinkman coupling. Moreover, we investigate the convergence of a global weak solution to a single equilibrium. The non-constant mobility and the low regularity of $\textbf{u}$ complicates the analysis. We establish the convergence of $\varphi$ by employing \L{}ojasiewicz-Simon inequality and subsequently deduce the convergence of $\textbf{u}$.

The main novelties of this work are:
\begin{itemize}
    \item We have estimated the separation scale explicitly in the degenerate mobility case and shown its dependence on the norm of the external forcing term, $\textbf{h}$, which can be treated as a control.
    \item We have assumed a growth condition on the mobility from above (see \textbf{[A4]}) and hence a minimum growth condition on the potential. Therefore our results apply to a larger class of potentials with singularity that dominates the degeneracy of mobility.
    Despite the challenges due to the low regularity of $\textbf{u}$ and non-constant mobility, we establish the convergence of a weak solution to an equilibrium state. 
\end{itemize}

The paper is organised as follows. Section 2 introduces the mathematical setting, assumptions, and preliminary estimates for the CHB system. Section 3 contains our main result on the separation property and its proof based on an adapted De Giorgi iteration method. We discuss the instant regularisation of the solution of the system. Further in section 4, we discuss the convergence to equilibrium solution of a single trajectory.

\section{Mathematical setup and preliminary results}
Let $\Omega \subseteq \mathbb{R}^2$ be an open bounded domain with a smooth boundary $\partial\Omega.$ Let H and V represent spaces $L^2(\Omega)$ and $H^1(\Omega)$ respectively. The norm in $H$ and $V$ are denoted by $\lVert \,\cdot\,\rVert$ and $\lVert \,\cdot\,\rVert_{V}$ respectively. The duality pairing between spaces $V$ and $V'$ is denoted by $\langle.,.\rangle$. We denote the closure of divergence-free space, $\mathcal{V } :=\{ u\in (C_0^\infty(\Omega)^2; \nabla. u = 0\}$ in $L^2(\Omega;\mathbb{R}^2)$ and $H^1(\Omega;\mathbb{R}^2)$ by $\mathbb{G}_{div}$ and $\mathbb{V}_{div}$ respectively. Then we have the following characterisation from \cite{TEM}, 
\begin{align*}
\mathbb{G}_{div} &= \{ \textbf{u}\in L^2(\Omega;\mathbb{R}^2): \text{div}(\textbf{u})=0, \textbf{u}.\textbf{n}|_{\partial \Omega}=0\},\\
 \mathbb{V}_{div} &= \{ \textbf{u}\in H^1_0(\Omega;\mathbb{R}^2): \text{div}(\textbf{u})=0\}.
\end{align*}
Let $\lVert \cdot \rVert$, $\lVert \cdot \rVert_{\mathbb{V}_{div}}$ be the norm in $\mathbb{G}_{div}$ and $\mathbb{V}_{div}$ respectively. Let $\mathbb{V}_{div}'$ be the dual space of $\mathbb{V}_{div}$ and $\langle.,.\rangle$ denote the duality pairing between $\mathbb{V}_{div}$ and $\mathbb{V}_{div}'.$  The collection of vector-valued functions with both its components in $X$ is denoted by $\textbf{X}$. 

For a Banach space $X$, $L^2_{loc}([0,\infty); X)$ denotes the space of functions whose norm in $X$ is  square integrable in any compact interval of time. The collection of translation-bounded functions consists of functions $ f\in L^2_{loc}([0,\infty); X)$ such that the following norm is finite.
\begin{align*}
    \int_{t}^{t+1}\|f(s)\|_{X} < C, \forall \; t \geq 0.
\end{align*}
where $C>0$ is independent of time $t$. In fact, the interval $[t, t+1]$ can be replaced by $[t,t+T_0]$ for any finite $T_0>0$.
%both the definitions are equivalent.
The norm in $L^2_{tb}([0,\infty);X)$ is defined as follows.
\begin{align*}
    \|f\|_{L^2_{tb}([0,\infty);X)}^2 = \sup_{t\geq 0}\int_{t}^{t+1}\|f(s)\|_{X}^2.
\end{align*}

We study the nonlocal CHB system with a singular-type potential and non-constant degenerate mobility under the following set of assumptions. 
\begin{enumerate}[font={\bfseries},label={A\arabic*.}]
    \item[{[N]}] The viscosity, $\nu$ and permeability, $\eta$ are positive constants.
    \item[{[J]}] $J \in W^{1,1}(\Omega)$ is an even function and $a(x) := \int_{\Omega} J(x-y)dy \geq 0 $ for almost all $ x\in \Omega.$
    And \\$\sup_{x\in \Omega}\int_{\Omega}|J(x-y)|dy  < \infty,$ $\sup_{x\in \Omega}\int_{\Omega}|\nabla J(x-y)|dy < \infty.$
\end{enumerate}
To treat the case of a singular potential, we need the following assumptions on the potential and mobility:
\begin{enumerate}[font={\bfseries},label={A\arabic*.}]
    \item[{[A1]}] The mobility $m\in C([-1,1])$ such that $m(s)\geq 0, \forall s \in [-1,1]$ and $m(s) =0$ iff $s=\pm 1.$ Further, there exists  $\epsilon_0 > 0$ such that $m$ is non-increasing in $[1-\epsilon_0,1]$ and non-decreasing in $[-1,-1+\epsilon_0].$
    \item[{[A2]}] There exists a constant $c_0$ such that the potential $F\in C^2(-1,1)$ satisfies $F''(s) \geq c_0, \forall s \in (-1,1)$ and $\lim_{s\rightarrow \pm 1} F'(s) \rightarrow \pm \infty$. 
    There exists an $\epsilon_0 > 0$ such that $F''$ is non-decreasing in $[1-\epsilon_0,1)$ and non-increasing in $(-1,-1+\epsilon_0]$.
    \item[{[A3]}] The potential $F$ satisfies $m^\alpha F''\in C([-1,1])$ for an $\alpha >1$. There exists $c_1 >0$ such that $m(s)^\alpha(F''(s) + a(x)) \geq c_1$, $\forall s\in(-1,1)$ and a.e $x \in \Omega$.
\end{enumerate}

 Note that instead of $m(s)F''(s) \in C[-1,1]$, here we assume $m(s)^\alpha F''(s) \in C[-1,1]$ for an $\alpha>1$, which means the singularity of $F''$ dominates the degeneracy of the mobility.  The logarithmic potential, \eqref{eq09} and degenerate mobility, \eqref{eq1.0} satisfies  assumptions \textbf{[A1]} -\textbf{[A3]}.

First we state the existence result for the system \eqref{eq1}-\eqref{eq07} under the above assumptions on mobility and potential.

\begin{theorem}\label{existence}[\textbf{Wellposedness of the nonlocal CHB system with singular potential and degenerate mobility}]
Let $T>0$, $\varphi_0\in L^\infty(\Omega)$ be such that $F(\varphi_0) \in L^1(\Omega)$, $\textbf{h}\in L^2(0,T;\mathbb{V}_{div}')$ and assumptions \emph{\textbf{[N]}}, \emph{\textbf{[J]}}, \emph{\textbf{[A1]}}- \emph{\textbf{[A3]}} hold. Further, let $M \in C^2(-1,1)$ be a function such that $m(s)M''(s) =1, \forall s\in [-1,1]$ with $ M(0)= M'(0)= 0$ and satisfies the condition $M(\varphi_0) \in L^1(\Omega)$. Then there exists a weak solution $(\varphi, \textbf{u})$ for the system, \eqref{eq1}-\eqref{eq07} of regularity,
\begin{align*}
\varphi  &\in C([0,T];H) \cap L^2(0,T;V) \cap H^1(0,T;V'),\\
\textbf{u} &\in L^2(0,T;\mathbb{V}_{div}),\\
m(\varphi)\nabla\mu &\in  L^2(0,T;H).
\end{align*} 
Moreover,
 \begin{align}\label{eq2.05}
    \varphi \in L^\infty(\Omega\times [0,T]),\,\, |\varphi(x,t)|\leq 1 \text{ a.e in } \Omega \times (0,T),
    \end{align}
 and $(\varphi, \textbf{u})$ satisfies the following weak formulation, $\forall \psi \in V$, $\textbf{v}\in \mathbb{V}_{div}$ and a.e $t\in (0,T)$
\begin{align}
    \langle \varphi', \psi \rangle + (\textbf{u}\nabla\varphi, \psi )+ ( m(\varphi)\nabla\mu, \nabla \psi) &= 0, \label{eq1.1} 
\\ \nu(\nabla \textbf{u}, \nabla\textbf{v}) + \eta (\textbf{u}, \textbf{v}) &= ( \mu\nabla\varphi, \textbf{v}) + \langle\textbf{h}, \textbf{v} \rangle, \label{eq1.2}\\ 
\varphi(0) &= \varphi_0, \text{ for a.e. }x \in \Omega.\label{eq1.3}
\end{align}
where $\mu = a\varphi- J \ast\varphi + F'(\varphi)$ for a.e $(x,t) \in \Omega\times (0,T)$.
Moreover, the solution satisfies the following energy inequality,
\begin{align}\label{eq2.5}
    \frac{d}{dt}\mathcal{E}(\varphi) + \|\sqrt{m(\varphi)}\nabla\mu\|^2 + c\|\textbf{u}\|_{\mathbb{V}_{div}}^2 \leq C\|\textbf{h}\|_{\mathbb{V}_{div}'}^2, 
\end{align} 
 where the energy $\mathcal{E}(\varphi)$ is given by, 
 \begin{equation}\label{eqE}
     \mathcal{E}(\varphi(t)):= \frac{1}{2}\int_\Omega\int_\Omega J(x-y)(\varphi(x)-\varphi(y))^2 dxdy + \int_\Omega F(\varphi)dx.
 \end{equation} 
\end{theorem}
\begin{proof}
   We give a brief sketch of the proof as it follows by similar arguments as in \cite[Theorem 2]{FRG}, \cite[Theorem 3.1]{CFG} and \cite[Proposition 2.1]{SDK}. In all three works, the potential and mobility satisfies the assumption \textbf{[A3]} for $\alpha= 1$, i.e, $m(s)F''(s) \geq c_1, \forall s\in (-1,1)$ for a constant $c_1 >0$. In contrast, here we have $m(s)^\alpha F''(s) \geq c_1>0, \forall s\in (-1,1)$ for an $\alpha >1$.
   
 We approximate the potential and mobility by functions, $F_\epsilon \in C^2[-1,1]$ and $m_\epsilon\in C[-1,1]$ as in \cite{FRG, CFG}, further by Galerkin approximation, we obtain a solution to the approximated problem, say $(\varphi_\epsilon, \textbf{u}_\epsilon)$ for each $\epsilon>0$. Now by considering test functions $\psi = \mu_\epsilon$, $\textbf{v} = \textbf{u}_\epsilon$ in \eqref{eq1.1}, \eqref{eq1.2}, we obtain the following estimate for the solution of the approximated problem. 
\begin{align*}
    \frac{d}{dt}\mathcal{E}_\epsilon(\varphi_\epsilon) + \|\sqrt{m_\epsilon(\varphi_\epsilon)}\nabla\mu_\epsilon\|^2 + c\|\textbf{u}_\epsilon\|_{\mathbb{V}_{div}}^2 \leq C\|\textbf{h}\|_{\mathbb{V}_{div}'}^2 
\end{align*} 
 where, $$\mathcal{E}_\epsilon(\varphi(t)):= \frac{1}{2}\int_\Omega\int_\Omega J(x-y)(\varphi(x)-\varphi(y))^2 dxdy + \int_\Omega F_\epsilon(\varphi)dx . $$
 The constant $C>0$ is independent of $\epsilon$. 
 By assumption \textbf{[A2]}, we have $F_\epsilon(s) \geq C_1|s|^2 -C_2$, for constants $C_1>0$ and $C_2 \geq 0$ independent of $\epsilon$. We can deduce an estimate for $L^\infty(0,T;H)$ norm of $\varphi_\epsilon$.
 
Considering a test function $M'(\varphi_\epsilon)$ in $\varphi_\epsilon$ equation analogous to \eqref{eq1.1} and by applying H\"olders and Young's inequalities to estimate the terms, we obtain,  
 \begin{align*}
     \frac{d}{dt}\int_\Omega M(\varphi_\epsilon) + \frac{c_0}{2}\|\nabla\varphi_\epsilon\|^2  &\leq C_{J,\Omega},
 \end{align*}
 where $C_{J, \Omega} >0$ is a constant that depends on $J$ and $\Omega$. Note that the entropy function, $M \geq 0$, since $M$ is strictly convex (by assumption \textbf{[A1]}) and $M(0) = M'(0) = 0$. Now by integrating the above inequality over $[0,T]$, and by the observation $M \geq 0$, we get an estimate for $ \varphi_{\epsilon} $ in $ L^2(0,T;V)$.
 
 Now by applying compactness arguments we deduce a weakly converging subsequence of $(\varphi_\epsilon, \textbf{u}_\epsilon)$, and a limit $(\varphi, \textbf{u})$, hence a solution of the system \eqref{eq1}- \eqref{eq07}. The pointwise bound for $\varphi$ can be derived as in \cite{FRG}.  
\end{proof}

In the following subsection we deduce the global in time estimates for the global weak solution to the system \eqref{eq1}-\eqref{eq07}. These estimates are crucial in the proof of separation property by De Giorgi iteration. The major challenge is due to the degeneracy of the mobility at pure phases. However, due to the assumption \textbf{[A3]}, we can control the degenerate mobility term.

\subsection{Uniform global estimates}

In this section we derive global in time estimates for the solution    obtained in theorem  \ref{existence} for the system with singular potential and degenerate mobility.  The estimates derived in this section will be used in the proof of the separation property in the following section. 
Here we additionally assume a growth condition on the mobility function. 
\begin{enumerate}[font={\bfseries},label={A\arabic*.}]
    \item[{[A4]}] There exists $C>0$ and  $\gamma >0$ such that 
    \begin{align*}
        m(s) \leq C(1-s^2)^{1/\gamma}, \text{ for }s\in [-1,1].
    \end{align*}
   where $ \alpha \leq \gamma< \infty$ and $\alpha$ is as in assumption [A3].
     %\item[{[A5]}] The entropy function, $M $ is such that $M \in C^3(-1,1), m(s)M''(s) =1$, $\forall s\in (-1,1)$, $M(0)= M'(0)= 0$ and satisfies the following growth condition,
    %\begin{equation} \label{MGrowthCondition}
     %   |M(s)| \leq \kappa s^2, \,\, \forall s\in (-1,1),
    %\end{equation}
    %for some $\kappa >0$.
\end{enumerate}

\begin{remark}\label{remark1}
    We remark that by assumption \textbf{[A4]} we have,
    \begin{align}\label{eq1.4}
        m(s)F''(s) = \frac{m(s)^\alpha F''(s)}{m(s)^{\alpha-1}} \geq c_1 M''(s)^{\alpha - 1} \geq \frac{c_1}{C^{\alpha-1}(1-s^2)^{(\alpha -1) /\gamma}} \geq c_2.
    \end{align}
    where $c_2 > 0$ is a constant that depends on $\gamma$ and $\alpha$. The above growth condition on the mobility, combined with the assumption \textbf{[A3]} we have a minimal growth condition on the potential.
\end{remark}

\begin{prop}
Let $\varphi_0\in L^\infty(\Omega)$ be such that $F(\varphi_0) \in L^1(\Omega)$,  and assumptions \textbf{[N]}, \textbf{[J]}, \textbf{[A1]}-\textbf{[A4]} hold and $M$ be a function as  in Theorem \ref{existence}. Additionally, assume  $\textbf{h}\in L^2([0,\infty);\mathbb{G}_{div})$. If we denote the unique solution of \eqref{eq1}-\eqref{eq07} by $(\varphi, \textbf{u})$, then there exists a constant, $C >0$ such that,
    \begin{align}\label{eq.21}
        \sup_{t \geq 0} \|\varphi(t)\|^2 + \int\limits_{t}^{t+1} \|\nabla\varphi\|^2 ds &\leq C,  \hspace{1cm}\forall t\geq 0.\\ 
        \sup_{t \geq 0} \|F(\varphi)\|^2 + \int\limits_{t}^{t+1}\big( \|\sqrt{F''(\varphi)}\nabla\varphi\|^2  +  \|\textbf{u}\|_{\mathbb{V}_{div}}^2  \big) ds &\leq C(1+ C_h),  \hspace{1cm}\forall t\geq 0.
    \end{align}
    Where $C> 0$ is a constant that depends only on $\varphi_0$ and the system parameters $m, J, |\Omega|$ and $C_h = C\|\textbf{h}\|_{L^2([0,\infty),\mathbb{V}_{div}')}^2$.
    %is a constant depending on the norm of the external forcing, $\|\textbf{h}\|_{L^2([0,\infty),\mathbb{V}_{div}')}^2$.
\end{prop}
\begin{proof}
From now on, $C$ denotes a generic constant that may vary from step to step. 
Consider the following equation obtained by testing \eqref{eq1} with $\varphi$,
\begin{align}\label{eq2.1}
    \frac{1}{2}\frac{d}{dt}\|\varphi\|^2 + (m(\varphi)(a(x)+ F''(\varphi))\nabla\varphi, \nabla\varphi)+ (m(\varphi)(\varphi\nabla a -\nabla J* \varphi),\nabla\varphi) = 0.
\end{align}
Applying H\"older's and Young's inequalities, we obtain, 
\begin{align*}
    |(m(\varphi)(\varphi\nabla a -\nabla J* \varphi),\nabla\varphi)| &\leq \|m\|_{L^\infty}(\|\nabla a\|_{L^\infty} + \|\nabla J\|_{L^1})|\Omega|^\frac{1}{2}\|\nabla\varphi\| \leq \frac{c_2}{2}\|\nabla\varphi\|^2 + C.
\end{align*}
Substitute back in \eqref{eq2.1}, use \eqref{eq1.4}, we get,
\begin{align}\label{eq2.2}
     \frac{1}{2}\frac{d}{dt}\|\varphi\|^2 + \frac{c_2}{2}\|\nabla\varphi\|^2 \leq C,
\end{align}
where $C>0$ is a constant which depends on $m$, $J$ and $|\Omega|$.
Note that, $ (\varphi, \varphi')= (\varphi-\Bar{\varphi}, \varphi')= (\varphi-\Bar{\varphi}, (\varphi-\Bar{\varphi})')$ since by conservation of mass we have $\Bar{\varphi}(t) = \Bar{\varphi_0}$. Applying Generalised Poincar\'e inequality, it follows that,
\begin{align*}
     \frac{1}{2}\frac{d}{dt}\|\varphi- \Bar{\varphi}\|^2 + \frac{c_2}{2}\|\varphi-\Bar{\varphi}\|^2  \leq C.
\end{align*}
An application of Gronwall's lemma yields a dissipative estimate, 
\begin{align*}
    \|\varphi(t)- \bar{\varphi}(t) \|^2 \leq \|\varphi_0- \Bar{\varphi_0}\|^2e^{-c_2t}+ \frac{2C}{c_2}, \,\hspace{.5cm} \forall t\geq 0.
\end{align*}
Therefore,
 \begin{align}\label{eq2.3}
    \text{ess}\sup_{t\geq 0}\|\varphi(t)\|^2 \leq C.
\end{align}
Where $C>0$ is a constant which depends on $m$, $J$ and $|\Omega|$.
Integrating \eqref{eq2.2} over $[t, t+1]$, where $t \geq 0$ and applying the estimate \eqref{eq2.3} leads to,
\begin{align}\label{eq2.4}
   \int\limits_{t}^{t+1} \|\nabla\varphi\|^2 dt  &\leq C,\hspace{.5cm} \forall t \geq 0,
\end{align}
Observe that the constant $C$ is uniform in $t$. Now by testing \eqref{eq1} with $M'(\varphi)$, we obtain,
\begin{align}
    \frac{d}{dt}\int_\Omega M(\varphi)dt + ((a+ F''(\varphi))\nabla\varphi, \nabla\varphi) + (( \varphi\nabla a - \nabla J* \varphi), \nabla\varphi) = 0
\end{align}
Applying H\"olders and Young's inequalities we get,
\begin{align}\label{eq2.13}
    \frac{d}{dt}\int_\Omega M(\varphi)dt + \|\sqrt{F''(\varphi)}\nabla\varphi\|^2 \leq C_1(1+ \|\nabla\varphi\|^2)
\end{align}
Observe that from the energy inequality \eqref{eq2.5}, we have, for $t\geq 0$,
\begin{align*}
    \mathcal{E}(\varphi(t)) \leq \mathcal{E}(\varphi_0) + \|\textbf{h}\|_{L^2(0,t; \mathbb{V}_{div}')}^2 \leq \mathcal{E}(\varphi_0) + \|\textbf{h}\|_{L^2(0,\infty; \mathbb{V}_{div}')}^2. 
\end{align*}
From the expression for energy $\mathcal{E}(\varphi(t))$, we can deduce the local in time integrability of $\|F(\varphi)\|_{L^1(\Omega)}$. This estimate and the assumption \textbf{[A3]}, leads to the local in time integrability of the entropy function, $M(\varphi)$. Our aim is to apply global in time estimate for $\|\sqrt{F''(\varphi)}\nabla\varphi\|$ from \eqref{eq2.13} by applying uniform Gronwall's inequality. Consider the expression for $\mathcal{E}(\varphi)$, it follows that,
\begin{align*}
    \int\limits_{t}^{t+1} \int_\Omega F(\varphi(s))ds &\leq \int\limits_{t}^{t+1} \mathcal{E}(\varphi(s))ds + C_J\int\limits_{t}^{t+1} \|\varphi\|^2,\\
    &\leq \mathcal{E}(\varphi_0) + \|\textbf{h}\|_{L^2(0,\infty,\mathbb{V}_{div}')}^2+ C_J\|\varphi\|_{L^\infty([0,\infty);H)}^2
\end{align*}
Further from assumption \textbf{[A3]},
\begin{align*}
    M''(s) \leq \frac{F''(s)}{c_1} \implies M(s) \leq \frac{F(s)}{c_1} + Cs 
\end{align*}
From the local integrability of $\|F(\varphi)\|_{L^1(\Omega)}$, it follows that,
\begin{align*}
    \int\limits_{t}^{t+1} \int_\Omega M(\varphi)dxds &\leq \mathcal{E}(\varphi_0) + \|\textbf{h}\|_{L^2(0,\infty,\mathbb{V}_{div}')}^2+ C_J\|\varphi\|_{L^\infty([0, \infty);H)}^2 + C \\ &\leq C(1+ \|\textbf{h}\|_{L^2([0,\infty),\mathbb{V}_{div}')}^2)
\end{align*}
Observe that the R.H.S term is a constant independent of time. Now by local in time integrability of the term $\int_\Omega M(\varphi)$, we apply uniform Gronwalls in \eqref{eq2.13}. It follows that,
\begin{equation}\label{eq2.14}
    \sup_{t\geq 0} \int_\Omega M(\varphi) dx \leq C(1+ \|\textbf{h}\|_{L^2([0,\infty),\mathbb{V}_{div}')}^2).
\end{equation}
Furthermore, by integrating over $[t,t+1]$, we obtain,
\begin{equation*}
    \int_{t}^{t+1}\int_\Omega F''(\varphi)|\nabla\varphi|^2 dxdt \leq C(1+ \|\textbf{h}\|_{L^2([0,\infty),\mathbb{V}_{div}')}^2).
\end{equation*}
we have by assumption \textbf{[A3]}, $F''(\varphi) \geq c_1 M''(\varphi)^\alpha$, thereby obtain,
\begin{equation}\label{eq2.15}
    \int_{t}^{t+1}\int_\Omega M''(\varphi)^\alpha|\nabla\varphi|^2 dxdt \leq C(1+ \|\textbf{h}\|_{L^2([0,\infty),\mathbb{V}_{div}')}^2).
\end{equation}

Testing \eqref{eq3} with $\textbf{u}$, and estimating each term by applying H\"older's, Young's and Korn's inequalities, we get, for any $t \geq 0$,
\begin{align*}
   \int\limits_{t}^{t+1}\|\textbf{u}\|_{\mathbb{V}_{div}}^2ds &\leq C(\int\limits_{t}^{t+1}\|\varphi\|_{V}^2 +\int\limits_{t}^{t+1}\|\textbf{h}\|_{\mathbb{V}_{div}'}^2) \leq C. 
\end{align*}
Therefore by \eqref{eq2.4}, we get the local integrability of $\|\textbf{u}\|_{\mathbb{V}_{div}}$ in time.

\end{proof}

%The final estimate, \eqref{eq2.12}, which is obtained from the Brinkman equation, is crucial in proving the separation property in the next section. Observe that the assumption $\textbf{h} \in L^2_{tb}([0,\infty); \mathbb{G}_{div})$ is not too restrictive since in most of the control problems the set of admissible control consists of functions that are almost everywhere bounded by constants. 

%\begin{remark}\label{remark2}
%    Note that the constant on the R.H.S of \eqref{eq2.12} is independent of time $t \geq 0$. Suppose the length of the interval, $r\geq 0$ and $\textbf{h}\in L^2_{tb}([0,\infty);H)$, we can express a uniform bound for $\int_{t}^{t+r}\|\textbf{h}\|^2 ds$ as a multiple of $r$. 
%    The above argument can be proved by considering the following two cases.
%    When $r\geq 1$, observe that, 
%    \begin{align*}
%        \int_{t}^{t+r}\|\textbf{h}\|^2 ds \leq \ceil*{r}\|\textbf{h}\|_{L^2_{tb}([0,\infty);H)} \leq 2r\|\textbf{h}\|_{L^2_{tb}([0,\infty);H)},
%    \end{align*}
%where $\ceil*{.}$ denotes the ceiling function. On the other hand if $r<1$, we choose the smallest natural number $n_0$ such that $n_0r>1$,
%Then we have,
%\begin{align*}
%    \int_{t}^{t+r}\|\textbf{h}\|^2 ds \leq \int_{t}^{t+n_0r}\|\textbf{h}\|^2 ds \leq 2rn_0\|\textbf{h}\|_{L^2_{tb}([0,\infty);H)}.
%\end{align*}
%A uniform global bound expressed as a multiple of the length of the interval plays a vital role in proving our main result.
%\end{remark} 

\section{Separation property}
The separation property for the CH equation with degenerate mobility is established using the iterative arguments of Alikakos Moser in \cite{GGF}, which is inspired by similar results proved in \cite{LSP, LPH} for the phase field system. In a recent work \cite{FRG26}, similar results for the CH equation with source terms is established using this iterative method. Analogously, one could exploit the fact that the velocity is divergence-free, allowing similar arguments to establish the separation property for our problem on the convective CHB system. However, we prove the result for a different class of potential and mobility. Our method is based on a uniform estimate derived for the solution of the system and the results can be easily extended for similar coupled system which includes Cahn-Hilliard Darcy (CHD), Cahn-Hilliard Hele-Shaw (CHHS) and Cahn-Hilliard/Navier-Stokes (CHNS) systems.

As mentioned before, we have slightly modified the assumptions in \cite{GGF} in which the degenerate mobility, $m$ with the potential $F$ satisfies the condition $mF'' \in C[-1,1]$. In contrast, we assume the singularity of the potential dominates over the degeneracy of the mobility, i.e, there exists an $\alpha>1$ such that the function, $m^\alpha F'' \in C[-1,1]$. The practically relevant logarithmic potential \eqref{eq09} with the mobility \eqref{eq1.0} satisfies our assumptions. 

We adopt De-Giorgi's method introduced in \cite{GGA}, where it has been used to prove the separation property for the nonlocal CH equation with constant mobility and singular type potential in dimension 2. The advantage of this method is that it does not require any condition on the third or higher order  derivative of the potential to prove the separation property. Moreover, this method allows us to explicitly calculate the separation scales, indicating how far it is from the pure phases.

In this work, we impose a growth condition on the mobility function near the pure phases $\pm 1$, which in turn implies a minimum growth condition on the potential. Our proof works for any potential $F$ and mobility $m$ satisfying $\textbf{[A1]}-\textbf{[A4]}$, and in particular for logarithmic potential \eqref{eq09} with mobility \eqref{eq1.0}. We point out that the external forcing term in the Brinkman equation contributes a control in the phase separation, which will be discussed in remark \ref{remark3.3}. Furthermore, considering the external forcing term $\textbf{h} \equiv \textbf{0}$, we have shown the convergence of single solution trajectory to a unique equilibrium solution.

Now we state our main result, which is the separation property for the nonlocal CHB system with a singular-type potential and degenerate mobility in dimension 2.

\begin{theorem}\label{separation}
Let $\varphi_0\in L^\infty(\Omega)$ be such that $\lVert\varphi_0\rVert_{L^\infty(\Omega)} \leq 1$,  $|\Bar{\varphi}_0| \leq 1- \lambda$, for some $\lambda >0$. Assume \emph{\textbf{[N]}}, \emph{\textbf{[J]}}, \emph{\textbf{[A1]}}-\emph{\textbf{[A4]}} are satisfied for $\alpha$, $\gamma$ such that $3\alpha-2\gamma> 2$. Then for any $\tau>0$, $\exists \; \delta\in (0,1)$ which depends on $\tau$ and $\varphi_0$ such that the unique solution to \eqref{eq1}-\eqref{eq07} satisfies
\begin{align}\label{eq3.01}
\sup_{t\geq \tau} \lVert\varphi(x,t)\rVert_{L^\infty(\Omega)} \leq 1-\delta.
\end{align}
\end{theorem}

Before proving the theorem, we recall a lemma established in \cite[Lemma 4.3]{GGA} and \cite[Lemma 3.8]{POI}. 
\begin{lemma}\label{lemma2}
  Let $\{y_n\}_{n\in \mathbb{N}\cup\{0\}}\subset \mathbb{R}^+$ satisfies,
  \begin{equation}\label{eql2}
      y_{n+1} \leq Cb^{n}y_{n}^{1+\epsilon},
  \end{equation}
  for some $C>0$, $b>1$ and $\epsilon>0$. Assume that $y_0 \leq C^{-1/\epsilon}b^{-1/\epsilon^2}$. Then,
   $$y_n \leq y_0b^{-n/\epsilon}, \hspace{1cm} n\geq 1.$$ 
   In particular, $y_n \rightarrow 0$ as $n \rightarrow \infty$.
\end{lemma}

%Before proving the theorem, we recall a lemma from \cite[Chapter 2, Lemma 5.7]{LZK}. A similar lemma is proved in \cite[Chapter 1, Lemma 4.2]{DIB} as a corollary of De Giorgi's iteration. 

%\begin{lemma}\label{lemma2}
%   Let $\{y_n\}_{n\in \mathbb{N}\cup\{0\}}, \{z_n\}_{n\in \mathbb{N}\cup\{0\}} \subset \mathbb{R}^+$ satisfies the system of recursive inequalities,
%   \begin{align}\label{eql2}
%       y_{n+1} &\leq Cb^{n}(y_{n}^{1+\epsilon}+ z_n^{1+\kappa}y_n^\epsilon),\\
%       z_{n+1} &\leq Cb^{n}(y_{n}+ z_n^{1+\kappa}),
%   \end{align}
%   where $C, b, \epsilon$ and $\kappa$ are fixed positive real numbers with $b > 1$.
 %  Then,
 %  $$y_n \leq \lambda b^{-n/d}, \hspace{1cm} z_n \leq (\lambda b^{-n/d})^{1/(1+\kappa)}.$$ 
%   where,
%   \begin{equation*}
%       d= \min\{ \epsilon, \frac{\kappa}{1+\kappa} \}, \hspace{1cm} \lambda = \min\{ (2C)^{-1/\epsilon}b^{-1/(\epsilon d)}, (2C)^{-(1+\kappa)/\kappa} b^{-1/ (\kappa d)}\}
%   \end{equation*}
%  as long as $y_0 \leq \lambda$ and $z_0 \leq \lambda^{1/(1+\kappa)}$. 
%In particular, $y_n, z_n \rightarrow 0$ as $n \rightarrow \infty$.
%\end{lemma}

\begin{proof}[Proof of Theorem \ref{separation}]

 The core argument proceeds as follows: We consider a finite time interval, $[0, T],$  choose a $\tilde \tau$ such that $0 < \tilde \tau < T $ and prove that there exists a $\delta>0$ such that if $t\in [T-\tilde \tau, T]$ then the measure of the set $\{ x\in \Omega: \varphi(x,t) \geq 1-\delta\}$ is zero. Since $T$ is arbitrary, we will be able to extend the result for all $ t \geq \tilde \tau$  and hence \eqref{eq3.01} will follow. 

For this, we define two sequences, one in time, say $t_n$ that increases to $T-\tilde \tau$  and another sequence in $\mathbb{R}$,  $(k_n) $ which increases to $1-\delta$. We will prove that the integral, $$\int\limits_{t_n}^{T}\int_{\{x\in \Omega: \varphi(x,t)\geq k_n\}} 1 dxdt,$$ goes to zero as $n$ tends to infinity using lemma \ref{lemma2} and by extending the time interval we obtain the separation property. 

First, we introduce the mathematical setup. Let $T>0$. Fix $\epsilon_0>0$ as in  \textbf{[A3]}. 
We choose $\tilde \tau$ and $\delta$ such that $T>3\tilde \tau$ and $ 0<\delta<\frac{\epsilon_0}{2}$. We will refine the choice of $\tilde \tau$ and $\delta$ later if necessary.
 We define two sequences: set $t_{-1} = T-3\tilde \tau,$ and define 
\begin{align*}
    t_n &= t_{n-1} + \frac{\tilde \tau}{2^n}, \,\,  n\geq 0. \\
k_n &= 1-\delta - \frac{\delta}{2^n}, \,\, n\geq 0.
\end{align*}
\\Observe that $t_n$ is an increasing sequence with bounds $t_{-1}< t_n < t_{n+1}< T-\tilde \tau, \,\, \forall n\geq 0$ and $t_n$ converges to $t_{-1}+ 2\tilde \tau = T-\tilde \tau $ as $n\rightarrow \infty$. Similarly $k_n$ is an increasing sequence with bounds $1-2\delta < k_n < k_{n+1} < 1-\delta, \,\, \forall n\geq 1$ and $k_n$ converges to $1-\delta$ as $n\rightarrow \infty$.
For $n\geq 0$, define
$$\varphi_n(x,t) = \text{max}\{ \varphi(x,t)-k_n, 0\} = (\varphi - k_n)_+.$$
Further, let $I_n := [t_{n-1}, T]$ and for  $ t\in I_n,$ define $$A_n(t) = \{ x\in \Omega | \varphi(x,t)- k_n \geq 0\},$$
and for $t\in [0,t_{n-1})$,  define, $ A_n(t):= \emptyset$. Clearly, $I_n \supseteq I_{n+1}, \forall n\geq 1$. Then $I_n $ becomes $ [T-\tilde \tau, T]$ as $n \rightarrow \infty$. In addition, $\forall n\geq 1, $ and $t\in I_{n+1}, \; A_{n+1}(t) \subseteq A_n(t).$

We set,
$$y_n = \int_{I_n}\int_{A_n(s)} 1 dxds, \,\, \forall n \geq 0. $$
In other words, $y_n$ represents the measure of the set $\{(x,t) \in \Omega\times [0,T] : \varphi(x,t) \geq k_n \text{ for some } t\in [t_{n-1}, T]\}$.
Our aim is to prove $y_n \rightarrow 0$ as $n\rightarrow \infty$. 
 For $n\geq 0$, consider cut-off functions, $\eta_n \in C^1(\mathbb{R})$ such that $|\eta_n(t)|\leq 1$, $|\eta_n'(t)|\leq \frac{2^{n+1}}{\tilde \tau}, \forall t $ and
\begin{align*}
\eta_n(t) &= \begin{cases}
               1, \text{ for } t\geq t_n,\\
               0, \text{ for } t\leq t_{n-1}.\\
            \end{cases}
\end{align*}
We will choose a test function $v=\varphi_n\eta_n^2$ and test it with \eqref{eq1} and integrate it over the interval $[0,t]$, for $t_{n} \leq t \leq T $. Note that $\eta_n = 0$ for $t\leq t_{n-1}$. 
\begin{align}\label{eq3.1}
    \int\limits_{t_{n-1}}^t \langle \varphi',\varphi_n\eta_n^2\rangle \; +  \int\limits_{t_{n-1}}^t (\textbf{u}\cdot\nabla\varphi, \varphi_n\eta_n^2)  &  \;+ \int\limits_{t_{n-1}}^t (m(\varphi)(F''(\varphi) + a )\nabla\varphi, \nabla\varphi_n)\eta_n^2  \nonumber\\&+  \int\limits_{t_{n-1}}^t ( m(\varphi)((\nabla a)\varphi - \nabla J*\varphi), \nabla\varphi_n)\eta_n^2 =0.
\end{align}
Now we estimate terms in \eqref{eq3.1}. Note that $\varphi_n(x) = 0, \forall x \not\in A_n $ and on $A_n$, $\varphi_n' = \varphi'$. Therefore we have,
\begin{align*}
    \int\limits_{t_{n-1}}^t \langle \varphi',\varphi_n\eta_n^2\rangle &= \frac{1}{2} \int\limits_{t_{n-1}}^t \frac{d}{dt}\lVert\varphi_n\eta_n\rVert^2 -  \int\limits_{t_{n-1}}^t \lVert\varphi_n(s)\rVert^2\eta_n\eta_n' ds\\
    &= \frac{1}{2}\lVert\varphi_n(t)\rVert^2 - \int\limits_{t_{n-1}}^{t}\lVert\varphi_n\rVert^2\eta_n\eta_n' ds. 
    \end{align*}
Since $t\geq t_{n}, \eta_n(t) = 1 $ and  $\eta_n(t_{n-1}) =0.$ The second term in the above inequality can be further estimated as follows using $ |\varphi_n| \leq 2\delta, \forall n$.
    \begin{align*}
    \int\limits_{t_{n-1}}^{t}\lVert\varphi_n\rVert^2\eta_n(s)\eta_n'(s) ds &\leq \frac{2^{n+1}}{\tilde \tau}\int\limits_{t_{n-1}}^{t}\int_{A_n}\varphi_n^2\eta_n(s) dxds  \leq \frac{2^{n+1}}{\tilde \tau}(2\delta)^2\int\limits_{t_{n-1}}^{t}\int_{A_n} 1 dxds \leq \frac{2^{n+1}(2\delta)^2}{\tilde \tau} y_n,    
  \end{align*}  
  To estimate the second term in \eqref{eq3.1}, using $ \nabla\cdot\textbf{u}=0 $ we get,
     \begin{align*}
    \int\limits_{t_{n-1}}^t (\textbf{u}\cdot\nabla\varphi, \varphi_n\eta_n^2) &= \frac{1}{2}\int\limits_{t_{n-1}}^t\int_{A_n}
     \textbf{u}\cdot\nabla(\varphi_n^2)\eta_n^2 = 0, 
\end{align*}
 by assumption \textbf{[A3]}, we estimate,
\begin{align*}
     \int\limits_{t_{n-1}}^t( m(\varphi)(F''(\varphi) + a )\nabla\varphi, \nabla\varphi_n)\eta_n^2 &= \int\limits_{t_{n-1}}^t\int_{A_n} m(\varphi)^\alpha(F''(\varphi) + a )(M''(\varphi))^{\alpha - 1}|\nabla\varphi_n|^2\eta_n^2 \\
     &\geq c_1 M''(1-2\delta)^{\alpha - 1}\int\limits_{t_{n-1}}^t\lVert\nabla\varphi_n\rVert^2\eta_n^2.
\end{align*}
The last inequality follows from the observation that if $0<\delta < \frac{\epsilon_0}{2}$, then by \textbf{[A1]}, $M''(1-2\delta) \leq M''(\varphi)$. Now by applying  H\"olders and Young's inequality we estimate,
\begin{align*}
     \int\limits_{t_{n-1}}^t&\big( m(\varphi)((\nabla a)\varphi - \nabla J*\varphi),\nabla\varphi_n\big)\eta_n^2 =  \int\limits_{t_{n-1}}^t \big(m(\varphi)((\nabla a)\varphi - \nabla J*\varphi)\eta_n, \nabla\varphi_n\eta_n\big)\\
     &\leq \frac{c_1 M''(1-2\delta)^{\alpha - 1}}{2}\int\limits_{t_{n-1}}^{t}\lVert\nabla\varphi_n\rVert^2\eta_n^2 + \frac{1}{2c_1 M''(1-2\delta)^{\alpha - 1}} \int\limits_{t_{n-1}}^t \int_{A_n} |m(\varphi)((\nabla a)\varphi - \nabla J*\varphi)|^2\eta_n^2\\
     &\leq \frac{c_1 M''(1-2\delta)^{\alpha - 1}}{2}\int\limits_{t_{n-1}}^{t}\lVert\nabla\varphi_n\rVert^2\eta_n^2 + \frac{\|m\|_{L^\infty}^2(\lVert\nabla a\rVert_{L^\infty(\Omega)}^2 + \lVert\nabla J\rVert_{L^1(\Omega)}^2)}{c_1 M''(1-2\delta)^{\alpha - 1}} \int\limits_{t_{n-1}}^t\int_{A_n}1 dxds\\
     &\leq \frac{c_1 M''(1-2\delta)^{\alpha - 1}}{2}\int\limits_{t_{n-1}}^{t}\lVert\nabla\varphi_n\rVert^2\eta_n^2 + \frac{\|m\|_{L^\infty}^2(\lVert\nabla a\rVert_{L^\infty}^2 + \lVert\nabla J\rVert_{L^1}^2)}{c_1 M''(1-2\delta)^{\alpha - 1}} y_n.
\end{align*}
Combining all the above estimates in \eqref{eq3.1} results in the following inequality.
\begin{align}\label{eq3.2}
  \frac{1}{2}\lVert\varphi_n(t)\rVert^2 + \frac{c_1 M''(1-2\delta)^{\alpha - 1}}{2}\int\limits_{t_{n-1}}^t\lVert\nabla\varphi_n\rVert^2\eta_n^2 \leq  \Big( \frac{2^{n+3}\delta^2}{\tilde \tau} + \frac{\|m\|_{L^\infty}^2(\lVert\nabla a\rVert_{L^\infty}^2 + \lVert\nabla J\rVert_{L^1}^2)}{c_1 M''(1-2\delta)^{\alpha - 1}}\Big) y_n.
\end{align}

Our aim is to derive an iterative condition for $y_n$ as in lemma \ref{lemma2}. To this end we define,
\begin{align*} 
P_n &= 2^{n+4} \max \bigl\{ \frac{\delta^2}{\tilde \tau}, \frac{C_J}{ M''(1-2\delta)^{\alpha - 1}} \bigr \} y_n.  
\end{align*}
where, $$C_J = \frac{\|m\|_{L^\infty}^2(\|\nabla a\|_{L^\infty}^2+ \|\nabla J\|_{L^1}^2)}{c_1}$$ 
is a constant that depends on the parameters $J$ and $m$. Thus, it follows from \eqref{eq3.2} that, 
\begin{equation} \label{phinBound}
\max_{t\in I_{n+1}} \frac{1}{2}\lVert\varphi_n(t)\rVert^2 \leq P_n,   \hspace{.5cm}\text{ and } \hspace{.5cm} \frac{c_1 M''(1-2\delta)^{\alpha - 1}}{2} \int_{I_{n+1}}\lVert\nabla\varphi_n\rVert^2 \leq P_n.
\end{equation}
Observe that on $A_n$, 
\begin{align*}
    \varphi_n &=\varphi - \big(1-\delta-\frac{\delta}{2^n}\big) = \varphi -\big(1-\delta -\frac{\delta}{2^{n+1}}\big)-\frac{\delta}{2^{n+1}}+ \frac{\delta}{2^n}\\ &=\varphi_{n+1} + \frac{\delta}{2^{n+1}} \geq \frac{\delta}{2^{n+1}},
\end{align*}
Since $\varphi_{n}\geq 0, \,\, \forall n$. This implies
 \begin{align}\label{eq3.091}
  \int_{I_{n+1}}\int_{A_{n+1}} |\varphi_n|^3 dxds \geq \Big(\frac{\delta}{2^{n+1}}\Big)^3  y_{n+1}.
\end{align}
Further, we have by applying H\"older's, Gagliardo-Nirenberg inequalities, and a comparison with \eqref{phinBound}, 
\begin{align}
\int_{I_{n+1}}\int_{A_{n+1}} |\varphi_n|^3 dxds &\leq \int_{I_{n+1}}\int_{A_n} |\varphi_n|^3 dxds\nonumber\\
    &\leq \Big(\int_{I_{n+1}}\int_{A_n} |\varphi_n|^4 dxds\Big)^{3/4}\Big(\int_{I_{n+1}}\int_{A_n} 1 dxds\Big)^{1/4}\nonumber\\
    &\leq C_{\Omega}\Big(\int_{I_{n+1}}\|\varphi_n\|^2\|\nabla\varphi_n\|^2 dxds\Big)^{3/4}\Big(\int_{I_{n}}\int_{A_n} 1 dxds\Big)^{1/4}\nonumber\\
    &\leq C_{\Omega} y_n^{1/4} \Big( \max_{t\in I_{n+1}} \lVert\varphi_n(t)\rVert^2 \int_{I_{n+1}}\lVert\nabla\varphi_n\rVert^2\Big)^{3/4} \nonumber\\
    &\leq C_{\Omega} y_n^{1/4} \Big( \frac{2^2}{c_1 M''(1-2\delta)^{\alpha - 1}}P_n^2\Big)^{3/4} \leq \frac{2^{3/2}C_{\Omega}}{c_1^{3/4}M''(1-2\delta)^{3(\alpha-1)/4}}P_n^{3/2}y_n^{1/4}.  \label{eq3.09}
\end{align} 
Note that, $C_{\Omega}$ is a constant that depends on $\Omega$. Now by comparing \eqref{eq3.091} and \eqref{eq3.09} we deduce that,
\begin{align*}
    y_{n+1} &\leq \frac{2^{3n+3}}{\delta^3} \frac{2^{3/2}C_{\Omega}}{c_1^{3/4}M''(1-2\delta)^{3(\alpha-1)/4}}P_n^{3/2}y_n^{1/4}, \\
    &\leq  \frac{2^{3n}2^{9/2}C_{\Omega}}{c_1^{3/4}} \frac{1}{\delta^3 M''(1-2\delta)^{3(\alpha-1)/4}}P_n^{3/2}y_n^{1/4}.
\end{align*}
Let us choose $\tilde{\tau}$ and $\delta$ satisfying, 
\begin{align}\label{eq3.009}
 \frac{\delta^2}{\tilde \tau} \geq  \frac{C_J}{ M''(1-2\delta)^{\alpha - 1}}.
\end{align}
This yields, 
\begin{align}
    y_{n+1} &\leq \frac{2^{9n/2}2^{21/2}C_{\Omega}}{c_1^{3/4}} \frac{1}{ M''(1-2\delta)^{3(\alpha-1)/4}\tilde{\tau}^{3/2}}y_n^{7/4}
   \label{3.10}
\end{align}
Our aim is to obtain $y_n \rightarrow 0$ as $n\rightarrow \infty$  using Lemma \ref{lemma2}, from the inequality \eqref{3.10}. This will prove our theorem.
Comparing with the condition in lemma \ref{lemma2}, we have,
\begin{align*}
    C = \frac{2^{21/2}C_{\Omega}}{c_1^{3/4}M''(1-2\delta)^{3(\alpha-1)/4}\tilde{\tau}^{3/2}};\hspace{.5cm} b = 2^{9/2}; \hspace{.5cm} \epsilon = 3/4 
\end{align*}
By invoking lemma \ref{lemma2}, we get if the initial data, $y_0$ satisfies,
\begin{align}\label{y0condition}
    y_0 &\leq C^{-1/\epsilon}b^{-1/\epsilon^2} = \frac{c_1 M''(1-2\delta)^{\alpha-1}\tilde{\tau}^2}{2^{22}C_{\Omega}^{4/3}},
\end{align}
Then $y_n \rightarrow 0$. It is left to prove the initial data, $y_0$ satisfies the condition \eqref{y0condition}. To this end, recall that $A_0 = \{ x\in \Omega : \varphi \geq 1-2\delta\}$
and by \textbf{[A1]}, \textbf{[A4]}, note that $M''$ is a non-decreasing function on the interval $[1-\epsilon_0, 1)$, we have $M''(1-2\delta)^{\alpha} \leq M''(\varphi)^{\alpha}$ for $\alpha>1$. Then using Poincar\'e inequality, we can write. 
\begin{align*}
    \int\limits_{T-3\tilde \tau}^{T}\int_{A_0} |\varphi -\Bar{\varphi}|^2dxds &\leq  \int\limits_{T-3\tilde \tau}^{T}\int_{\Omega} |\varphi -\Bar{\varphi}|^2dxds  \leq C\int\limits_{T-3\tilde \tau}^{T}\int_{\Omega} |\nabla\varphi|^2 dxds \\
    &\leq \frac{C}{M''(1-2\delta)^{\alpha}}\int\limits_{T-3\tilde \tau}^{T}\int_{\Omega} M''(\varphi)^{\alpha}|\nabla\varphi|^2dxds \\
     &\leq \frac{C}{M''(1-2\delta)^{\alpha}}\int\limits_{T-3\tilde \tau}^{T}\int_{\Omega} M''(\varphi)^{\alpha}|\nabla\varphi|^2dxds \leq \frac{3C\tilde{\tau}(1+C_h)}{M''(1-2\delta)^{\alpha}}
\end{align*}
The last inequality is obtained by applying the global estimate \eqref{eq2.15} for $M''(\varphi)^{\alpha}|\nabla \varphi|^2$.  We denote $C_h >0$, a constant which depends on $L^2([0,\infty); \mathbb{V}_{div}')$ norm of $\textbf{h}$ to indicate the contribution of the external source term in the separation. By definition, we have on $A_0$, $\varphi \geq 1-2\delta$ and hence, $\varphi-\Bar{\varphi}_0 \geq 1-2\delta-\Bar{\varphi}_0 \geq \lambda -2\delta >0$. Therefore, from the above inequality it follows that,
\begin{align}\label{eq3.15}
   y_0 = \int\limits_{T-3\tilde \tau}^{T}\int_{A_0} 1 dxds &\leq \frac{3C\tilde{\tau}(1+C_h)}{M''(1-2\delta)^{\alpha}(\lambda -2\delta)^2}.
\end{align}
 Therefore, the $y_0$ condition, \eqref{y0condition} is satisfied if,
 \begin{align}\label{d4}
\frac{3C\tilde{\tau}(1+C_h)}{M''(1-2\delta)^{\alpha}(\lambda -2\delta)^2} &\leq  \frac{c_1 M''(1-2\delta)^{\alpha-1}\tilde{\tau}^2}{2^{22}C_{\Omega}^{4/3}} , \nonumber\\
\implies \frac{3C(1+C_h)2^{22}C_{\Omega}^{4/3}}{c_1 M''(1-2\delta)^{2\alpha -1}(\lambda -2\delta)^2} &\leq  \tilde{\tau}.
\end{align}
We need to find $\delta, \Tilde{\tau} >0 $ satisfying the conditions\eqref{eq3.009} and \eqref{d4} simultaneously. i.e,
 \begin{align}\label{taucondn}
     \frac{3C(1+C_h)2^{22}C_{\Omega}^{4/3}}{c_1 M''(1-2\delta)^{2\alpha -1}(\lambda -2\delta)^2} \leq  \tilde{\tau} \leq \frac{\delta^2M''(1-2\delta)^{\alpha - 1}}{C_J}.
 \end{align}
To prove the above range of $\tilde{\tau}$ is non-empty, observe that by assumption \textbf{[A4]}, we have, $M''(s)^\gamma \geq \frac{1}{C(1-s^2)}$. Therefore, 
\begin{align*}
   \frac{\frac{\delta^2M''(1-2\delta)^{\alpha - 1}}{C_J}}{\frac{3C(1+C_h)2^{22}C_{\Omega}^{4/3}}{c_1 M''(1-2\delta)^{2\alpha -1}(\lambda -2\delta)^2}} &= \frac{c_1 \delta^2 M''(1-2\delta)^{3\alpha - 2} (\lambda -2\delta)^2}{3C(1+C_h)2^{22}C_{\Omega}^{4/3}C_J} \\
   &\geq \frac{c_1}{3C(1+C_h)2^{22}C_{\Omega}^{4/3}C_J}\frac{\delta^2 (\lambda -2\delta)^2}{ (4\delta(1-\delta))^{(3\alpha-2)/\gamma}} \\
   &\geq \frac{c_1}{3C(1+C_h)2^{22}C_{\Omega}^{4/3}C_J}\frac{\delta^2 (\lambda -2\delta)^2}{ (4\delta(1-\delta))^{(3\alpha-2)/\gamma}}  
    \sim O\Big(\frac{1}{\delta^{\frac{3\alpha-2}{\gamma}-2}}\Big). 
\end{align*}
 %&\geq \frac{c_1}{3C2^{22}C_{\Omega}^{4/3}C_J}\frac{\delta^2 (\lambda -2\delta)^2}{ (4\delta(1-\delta))^{(3\alpha-2)/\alpha}} \sim O\Big(\frac{1}{\delta^{1-\frac{2}{\alpha}}}\Big). 
%Therefore we can choose a $\delta>0$ such that the R.H.S  is greater than to 1,   
Therefore, if $ 3 \alpha - 2 \gamma > 2$ then for $\delta$ sufficiently small, we can choose a $\tilde{\tau}$ in the range given in \eqref{taucondn}.
% under the following condition on $\alpha$, $\gamma$ that
%\begin{equation}
%    1-\frac{2}{\alpha}> 0 \implies \alpha >2. 
%\end{equation}
For example, if $\gamma = \alpha$, then for $\alpha >2$ we can guarantee that there exists a choice $\Tilde{\tau}$ satisfying \eqref{taucondn}.
%\begin{align*}
 %   1-\frac{2}{\alpha} >0 \implies \alpha >2.
%\end{align*}
%Therefore, the conditions \eqref{eq3.009} and \eqref{d4} are compatible for $\alpha > 2$. 
Thus, for a choice of $\delta, \tilde \tau$ satisfying \eqref{eq3.009} and \eqref{d4}, we can verify inequalities \eqref{3.10}, \eqref{y0condition}, to obtain,
\begin{align*}
    y_n \leq y_0 2^{-6n}.
\end{align*}
Therefore, by applying lemma \ref{lemma2} we get $y_n \rightarrow 0$ as $ n\rightarrow \infty$ and hence
$$ |\varphi(x,t)| \leq 1-\delta, \hspace {.5cm}\text{ for a.e } x\in \Omega \text{ and } t\in [T-\tilde \tau, T].$$

In particular, consider $T= \frac{\tau}{2}+ 3\tilde \tau$, by repeating similar arguments for $(\varphi+ (-1+\delta))^-$, we can conclude that,
$$ -1+ \delta \leq \varphi(x,t) \leq 1-\delta, \hspace{.5cm} \text{ for a.e } x\in\Omega \text{ and } t\in [\frac{\tau}{2} + 2\tilde \tau, \frac{\tau}{2} + 3\tilde \tau].$$
Thus starting from $ \frac{\tau}{2} + 2 \tilde \tau$ we can repeat the argument for any interval of length $\tilde \tau$. Hence we choose $\tilde \tau \leq \frac{\tau}{4}$, and a corresponding $\delta$ so that the instantaneous separation property \eqref{eq3.01} follows.

\end{proof}

\begin{remark}\label{remark3.4}
   Consider the logarithmic potential given in \eqref{eq09} and mobility \eqref{eq1.0} with $2 <p <\infty$. Observe that the assumptions \textbf{[A3]}, \textbf{[A4]} holds for $\alpha = \gamma = p$. Choose the entropy function $M$ such that,
   \begin{equation*}
       M''(r) = \frac{1}{(1-r^2)^{1/p}} \leq \frac{1}{\sqrt{1-r^2}},\hspace{.5cm} \text{ for }r\in (-1,1).
   \end{equation*}
It follows that,
\begin{align*}
       M'(r) &= \int_{0}^{r} \frac{1}{(1-\tau^2)^{1/p}}d\tau \leq \text{sin}^{-1}(r), \\ M(r) &=  \int_{0}^{r}\int_{0}^{\tau} \frac{1}{(1-s^2)^{1/p}}ds d\tau \leq r\text{sin}^{-1}(r) + \sqrt{1-r^2} -1 \leq Cr^2
   \end{align*} 
   Moreover,
   \begin{align*}
        M''(1-2\delta) = \frac{1}{(4\delta(1-\delta))^{1/p}} \sim  O(\frac{1}{\delta^{1/p}}).
   \end{align*}
Note that the conditions \eqref{eq3.009} and \eqref{d4} for this specific case reduces to,
 \begin{align} \label{taueqn}
 \frac{3C(1+C_h)2^{22}C_{\Omega}^{4/3}}{c_1 M''(1-2\delta)^{2p-1}(\lambda -2\delta)^2} &\leq  \tilde{\tau} \leq \frac{\delta^2 M''(1-2\delta)^{p-1}}{C_J} 
\end{align}
Further observe that,
\begin{align*}
     \frac{\delta^2 M''(1-2\delta)^{p-1}}{C_J} \frac{c_1 M''(1-2\delta)^{2p-1}(\lambda -2\delta)^2}{3C(1+C_h)2^{22}C_{\Omega}^{4/3}} &\geq 
      \frac{c_1(\lambda -2\delta)^2}{3C(1+C_h)2^{22}C_{\Omega}^{4/3}C_J} \delta^2 M''(1-2\delta)^{3p-2} \sim O(\frac{1}{\delta^{1-\frac{2}{p}}})
\end{align*}
Therefore, for $p>2$, there exist a choice of $\delta>0$, such that the range of $\Tilde{\tau}$ in \eqref{taueqn} is non-empty.  
\end{remark}

\begin{remark}\label{remark3.3}
 Note that using the growth condition on $m$ given in \textbf{[A4]}, we can deduce  a condition on minimum growth of $F''$ from \textbf{[A3]}. However there is no upper bound for the growth of $F''$. 
 Furthermore, we have a constant $C_h> 0$ in the $\Tilde{\tau}$, $\delta$ condition \eqref{taucondn} which depends on the $L^2([0,\infty); \mathbb{V}_{div}')$ norm of $\textbf{h}$, the external source term in the velocity equation. When the $\|\textbf{h}\|_{L^2([0,\infty); \mathbb{V}_{div}')}$ is sufficiently large we have $\textbf{h}$ contributes to the separation parameters $\tilde{\tau}$ and $\delta$ and hence the contribution of the Brinkman coupling.  
\end{remark}

By exploiting the separation property, one can obtain regularised solution of the system.
\begin{corollary}
Assume all the hypotheses of Theorem \ref{separation} hold true. Then for any $\tau > 0$, there exists a constant $C_\tau > 0$ such that,
\begin{align}
    \|\mu\|_{L^\infty(\Omega)} + \|\nabla\mu\|_{L^2(t, t+1; H)} \leq C_\tau, \hspace{.5cm} \forall t\geq \tau.
\end{align}
\end{corollary}
Thanks to the separation property \eqref{eq3.01}, $F \in C^2(-1,1)$, we obtain $L^\infty(\tau, \infty; L^\infty(\Omega))$ regularity of $\mu$. In addition, the energy function, $\mathcal{E}(\varphi(t))$ defined in \eqref{eqE} is locally integrable in time. By applying uniform Gronwall lemma in the energy estimate \eqref{eq2.5} leads to the local in time regularity of $\nabla\mu$ (See \cite[Corollary A.5.]{GGF} for a similar result on the CH equation).

Now we discuss the instant regularization of the solution due to the separation property. Since the potential is regular in $(-1+\delta, 1-\delta)$, and the degeneracy of the mobility is removed, the solution achieves better regularity. Below we state a theorem on the regularity of the solution global in time.

\begin{theorem}\label{regularity}
    Let all the hypotheses of theorem \ref{separation} hold true. In addition assume $\textbf{h} \in L^\infty(0,\infty; \mathbb{V}_{div}')$. Then the solution of the system \eqref{eq1}-\eqref{eq07} satisfies,
    \begin{align}
        \|\varphi\|_{L^\infty(\tau, \infty; H)\cap L^2(\tau, \infty ;V)} &\leq C, \label{eq3.19}\\
        \|\varphi'\|_{L^2(\tau, \infty; V')} &\leq C,\label{eq3.20}\\
        \|\textbf{u}\|_{L^\infty(\tau, \infty; \mathbb{V}_{div})} &\leq C\label{eq3.21}\\
        \|\mu\|_{L^2(\tau, \infty; V) } &\leq C.\label{eq3.22}
    \end{align}
     for every $0 <\tau  \leq t $.
    % \begin{align}
     %   \|\varphi\|_{L^\infty(\tau, \infty; H)\cap L^2(t,t+1;H^2(\Omega))} &\leq C, \label{eq3.19}\\
      %  \|\varphi'\|_{L^\infty(\tau, \infty; V')\cap L^2(t,t+1;H)} &\leq C,\label{eq3.20}\\
      %  \|\textbf{u}\|_{L^2(\tau, \infty; \mathbb{V}_{div}) \cap H^1(t,t+1; \mathbb{V}_{div})} &\leq C\label{eq3.21}\\
      %  \|\mu\|_{L^2(\tau, \infty; V) \cap L^2(t, t+1; H^2(\Omega))} &\leq C.\label{eq3.22}
    %\end{align}
\end{theorem}
\begin{proof}[\textit{Outline of the proof}:]
The proof of the theorem follows by using the separation property and employing the estimates derived for the strong solution of the system in \cite[Theorem 2.3]{SDK}. We have an energy estimate \eqref{eq2.5} satisfied by the solution. Further observe that the energy function is local in time integrable. Therefore by applying the uniform Gronwall's inequality we obtain, 
global in time estimates for $\varphi \in L^\infty(\tau, \infty; H)$, $\nabla \mu \in L^2(\tau, \infty; H)$, $\textbf{u} \in L^2(\tau, \infty; \mathbb{V}_{div})$ immediately from the energy estimate. Subsequently, from \eqref{eq2.2}, \eqref{eq1.1}  $\varphi \in L^2(\tau, \infty; V)$ and $\varphi' \in L^2(\tau, \infty; V')$. 

Additionally, if we assume $\textbf{h} \in L^\infty(0, \infty; \mathbb{V}_{div}')$, by testing \eqref{eq2} with $\textbf{u}$ and estimating each term by applying H\"older's and Young's inequalities, we get,
\begin{align*}
    \|\textbf{u}(t)\|_{\mathbb{V}_{div}}^2 &\leq C(\|\nabla a\|_{L^\infty}+ \|\nabla J\|_{L^1})\|\varphi\|_{L^\infty(\Omega)}^2\|\textbf{u}\|_{\mathbb{G}_{div}} + C\|\textbf{h}\|_{\mathbb{V}_{div}'}\|\textbf{u}\|_{\mathbb{V}_{div}}\\
    &\leq \frac{1}{2}\|\textbf{u}\|_{\mathbb{V}_{div}}^2 + C\big((\|\nabla a\|_{L^\infty}+ \|\nabla J\|_{L^1})\|\varphi\|_{L^\infty(\Omega)}^4 + \|\textbf{h}\|_{\mathbb{V}_{div}'}\big).
\end{align*}
An application of separation property leads to the estimate,  \eqref{eq3.21}. 

%Moreover, we differentiate \eqref{eq1} w.r.t the time variable formally and obtain estimates for $\varphi'$ by testing with $\mathcal{N}(\varphi - \Bar{\varphi})'$. By applying uniform Gronwall's inequality we get, $\varphi' \in L^\infty(\tau, \infty; V')$. Further, we get the local in time integrability $\varphi' \in L^2(t, t+1; H)$, $\forall t\geq \tau$. Consequently, $\mu \in L^2(t, t+1; H^2(\Omega))$, $\varphi \in L^2(t, t+1; H^2(\Omega))$. Using the estimates for $\varphi'$, we deduce $\textbf{u}' \in L^2(t,t+1; \mathbb{V}_{div})$ following the arguments in \cite[Theorem 2.3]{SDK}. This completes the proof. 
\end{proof}

The separation property is essential to study the long time behaviour of the solution. In the following section we study the convergence of the solution of the system to an equilibrium steady state solution when the external force $h \equiv 0$ a.e in $\Omega \times (0,T)$. We use the \L{}ojasiewicz-Simon inequality to derive an estimate for the $L^1(\tau,\infty;H)$ estimate of $\nabla\mu$ inspired by other previous works. Finally, we show the convergence of solution of the nonlocal CHB system to an equilibrium state. 

Similar result for the local CHB system with a constant mobility is studied in \cite{BCG}. However, the long time behaviour of the nonlocal CHB system with non-constant mobility is challenging due to the low regularity of the velocity.

\section{Convergence to equilibrium}

Consider the admissible class of initial data defined by,
\begin{align}
    \mathcal{H}_k := \{ \varphi \in  L^\infty(\Omega) : \|\varphi\|_{L^\infty(\Omega)} \leq 1, F(\varphi) \in L^1(\Omega), |\Bar{\varphi}| = k \}
\end{align}
for $k \in [0,1)$. Let $\varphi_0 \in \mathcal{H}_k$ and $(\varphi, \textbf{u})$ be a solution of the system \eqref{eq1}-\eqref{eq07} corresponding to $\varphi_0$.
We define the set of all limit points associated with $\varphi$ by,
\begin{align}
    \omega(\varphi) := \{ \Tilde{\varphi} \in \mathcal{H}_k : \exists t_n \rightarrow \infty \text{ such that } \varphi(t_n) \rightarrow \Tilde{\varphi} \text{ in } H \}. 
\end{align}
Note that, 
\begin{equation}\label{eq4.03}
\lim_{t\rightarrow \infty} \text{dist}(\varphi(t), \omega(\varphi))_{H} \rightarrow 0.
\end{equation} 
Observe that by the compact embedding,
$$\varphi \in L^2(\tau,\infty;V) \cap H^{1}(\tau, \infty;V') \hookrightarrow \hookrightarrow C([\tau, \infty); H),$$ 
and estimates \eqref{eq3.19}, \eqref{eq3.20}, we obtain $\|\varphi(t)\|$ is uniformly bounded for all $t\geq \tau$. Now by a compactness argument we get $t_n$ such that $\varphi(t_n)$ converges to a limit $\Tilde{\varphi} \in H$. Therefore, the set $\omega(\varphi)$ is non-empty. Now we characterise the elements of $\omega(\varphi)$. For, we define an equilibrium point of the system \eqref{eq1}-\eqref{eq07}.

\begin{definition}\label{def4}
    We say $(\varphi_\infty, \textbf{0})$ is an equilibrium point of the system \eqref{eq1}-\eqref{eq07}, if $\varphi_\infty \in \mathcal{H}_k \cap V$ satisfies the following stationary Cahn-Hilliard system,
    \begin{align}
      a\varphi_\infty - J*\varphi_\infty + F'(\varphi_\infty) &= \mu_\infty,\label{eq4.3} \\
        \int_\Omega \varphi_\infty dx &= \int_\Omega\varphi_0 dx.\label{eq4.5}
    \end{align}  
 where $\mu_\infty \in \mathbb{R}$ is a constant.
\end{definition}
The existence as well as separation property of solution of the system \eqref{eq4.3}- \eqref{eq4.5} follows as in \cite{DGG}. 
We define a set,
 $$\mathcal{S} := \{ \varphi_\infty\in \mathcal{H}_{k} \cap V : \varphi_\infty \text{ satisfies } \eqref{eq4.3}- \eqref{eq4.5}\}, $$

We prove that the limit set $\omega(\varphi) \subseteq \mathcal{S}$. Let $\Tilde{\varphi} \in \omega(\varphi)$, then there exists a sequence, $t_n \geq 0, t_n \rightarrow \infty$ such that $\varphi(t_n)\rightarrow \Tilde{\varphi}$ as $n \rightarrow \infty$. Define a sequence of trajectories, $(\varphi_n(t), \textbf{u}_n(t)) := (\varphi(t_n+t), \textbf{u}(t_n+t))$ for $ n\in \mathbb{N}, t\geq 0$. By exploiting the estimates \eqref{eq3.19}, \eqref{eq3.20}, \eqref{eq3.21} and Lions- Magenes lemma, we obtain a limit $\varphi^* \in C([0,\infty); H)$, $\textbf{u}^* \in L^2(0, \infty; \mathbb{V}_{div})$ such that,
\begin{align*}
     \varphi_n &\rightharpoonup \varphi^* \text{ in } L^2(0, \infty; H),\\
     \varphi_n' &\rightharpoonup (\varphi^*)' \text{ in } L^2(0, \infty; V'),\\
     \textbf{u}_n &\rightharpoonup \textbf{u}^* \text{ in } L^2(0,\infty; \mathbb{V}_{div}),
\end{align*}
as $n$ tends to $\infty$. Observe that $(\varphi^*, \textbf{u}^*)$ solves the system \eqref{eq1}-\eqref{eq07} with $\varphi^*(0) = \lim_{n\rightarrow \infty} \varphi_n(0) = \tilde{\varphi} $. We have, 
$$\lim_{n\rightarrow \infty} \mathcal{E}(\varphi_n(t)) \rightarrow \mathcal{E}(\varphi^*).$$
Since the energy, $\mathcal{E}(\varphi(t))$ is non-increasing, we have a limiting value, $\mathcal{E}_\infty \in \mathbb{R}$ such that, 
$$\lim_{t\rightarrow \infty} \mathcal{E}(\varphi(t)) = \mathcal{E}_\infty = \lim_{n\rightarrow \infty} \mathcal{E}(\varphi_n(t)) = \mathcal{E}(\varphi^*),$$
which implies that $\mathcal{E}(\varphi^*) = \mathcal{E}_\infty \in \mathbb{R}$, a constant. 
Therefore  by the energy estimate \eqref{eq2.5} described in theorem \ref{existence}, we get,
$$ \mathcal{E}(\varphi^*(t)) + \int\limits_{s}^{t}\|\sqrt{m(\varphi^*)}\nabla\mu^*\|^2 d\tau +  \int\limits_{s}^{t} \|\textbf{u}^*\|_{\mathbb{V}_{div}}^2 d\tau = \mathcal{E}(\varphi^*(s)) = \mathcal{E}_\infty, \text{ for } 0\leq s<t< \infty, $$

Using separation property of $\varphi^*$, $m(\varphi^*)$ is non degenerate we conclude $\nabla\mu^*(t) = \textbf{0} \in \textbf{H}$ for a.e $t\in (0,\infty)$ using the above energy estimate. Further observe, $\textbf{u}^* =0$ in $\mathbb{V}_{div}$ for a.e $t\in (0, \infty)$. Comparing with the estimate, 
$$\|(\varphi^*)'\|_{V'} \leq \|\nabla \mu^*\| + \|\textbf{u}^*\|$$ 
leads to $(\varphi^*)' = 0$ in $V'$ for a.e $t\in (0, \infty)$.   Hence $\varphi^*(t) = \varphi^*(0) = \Tilde{\varphi}$, which is independent of time.  Therefore $\varphi^*$ is a solution of the stationary CH equation \eqref{eq4.3}- \eqref{eq4.5}. Hence, $\omega(\varphi) \subseteq \mathcal{S}$. 

Now we prove that the set $\omega(\varphi)$ is a singleton set.

Using the \L{}ojasiewicz-Simon inequality, we derive an estimate for $L^1(\tau, \infty ; H)$ norm of $\nabla \mu$ which is non-trivial. This is achieved by the following result from \cite[Proposition 6.2]{DGG}.
\begin{lemma}[(\L{}ojasiewicz Simon inequality)]\label{lemma4.1}
Let $\mathcal{\textbf{P}}: H \rightarrow L^2_0$ be projection map. Assume $F$ satisfies \textbf{[A2]} and additionally, $F$ is real analytic in $(-1,1)$. Let $\varphi \in \mathcal{H}_k \cap V$ be such that $|\varphi(x)| \leq 1-\delta$, $x\in \Bar{\Omega}$ for some $\delta >0$ and $\varphi_\infty \in \mathcal{S}$. Then there exists $\theta \in (0,\frac{1}{2}), C>0, \beta > 0$ such that
 \begin{align*}
     |\mathcal{E}(\varphi) - \mathcal{E}(\varphi_\infty)|^{1-\theta} \leq C\|\mathcal{\textbf{P}}(a\varphi -J*\varphi +F'(\varphi))\|_{V'},
 \end{align*}
 whenever $\|\varphi - \varphi_\infty\| \leq \beta$. 
\end{lemma}

Let $\varphi_\infty \in \omega(\varphi)$. Recall that $\mathcal{E}(\varphi(t)) \rightarrow \mathcal{E}(\varphi_\infty)$.  We assume w. l. o. g. that $\mathcal{E}(\varphi(t))> \mathcal{E}(\varphi_\infty)$. Otherwise, if there exist $\tilde{t}$ such that $\mathcal{E}(\varphi(\tilde{t})) = \mathcal{E}(\varphi_\infty)$, then $\varphi(t) = \varphi(\tilde{t}) =\varphi_\infty$, $\forall t\geq \tilde{t}$. We fix $\theta$ and $\beta$ as in lemma \ref{lemma4.1}. Then there exist a $t_0$ such that $\|\varphi(t) -\varphi_\infty\| \leq \beta$ for all $t\geq t_0$. By theorem \ref{separation} we have $|\varphi(x)| \leq 1-\delta$ for some $\delta >0$. We have by lemma \ref{lemma4.1}, for $t\geq t_0$,
\begin{align*}
    (\mathcal{E}(\varphi)-\mathcal{E}(\varphi_\infty))^{1-\theta} &\leq C\|\mathcal{\textbf{P}}(a\varphi-J*\varphi+F'(\varphi))\| \leq \|\nabla\mu\|,
\end{align*}
using generalised Poincar\'e inequality.
%If $\|\varphi - \varphi_\infty\| \leq \beta$, for a constant $\beta>0$, which will follow from \eqref{eq4.03}, we have,
From the energy inequality \eqref{eq2.5} we have,
\begin{align*}
    \frac{-d}{dt}(\mathcal{E}(\varphi)- \mathcal{E}(\varphi_\infty))^\theta &=  \theta(\mathcal{E}(\varphi)-\mathcal{E}(\varphi_\infty))^{\theta-1} (\frac{-d\mathcal{E}(\varphi)}{dt}) \\
    &\geq \frac{\theta(\|\textbf{u}\|^2 + \|\sqrt{m(\varphi)}\nabla\mu\|^2)}{C\|\nabla\mu\|} \geq C_1\|\nabla\mu\|,
\end{align*}
where the constant $C_1$ is independent of $t$. 
Integrating the above inequality over $(t_0, \infty)$, and since $\mathcal{E}$ is non-increasing we get for sufficiently large $t_0 >0$, $\mu \in L^1(t_0, \infty; V)$. Subsequently obtain, $\varphi' \in L^1(t_0, \infty; V')$. Furthermore, 
\begin{align}\label{eq4.6}
    \varphi(t) = \varphi(t_n) + \int\limits_{t_n}^{t}\varphi'(s)ds \rightarrow \varphi_\infty \text{ in } V',  \text{ as } t\rightarrow \infty.
\end{align}
Now by uniqueness of the limit, we conclude that $\omega(\varphi)$ contains a single element \cite[Theorem 1.2]{CGGG}. 

\textbf{Convergence of $u$ :}
Consider the velocity equation, 
\begin{align}\label{eq4.7}
    \nu(\nabla \textbf{u}, \nabla\textbf{v}) + \eta (\textbf{u}, \textbf{v})  &= (\mu\nabla \varphi -\mu_\infty\nabla \varphi_\infty, \textbf{v}), \forall \textbf{v}\in \mathbb{V}_{div}.
\end{align}
Since $\mu_\infty$ is a constant  term $\mu_\infty\nabla\varphi_\infty = \nabla(\mu_\infty \varphi_\infty)$ . Observe that,
$$\mu\nabla \varphi -\mu_\infty\nabla \varphi_\infty = \frac{1}{2} \nabla a(\varphi^2 - \varphi_\infty^2) - \nabla J *\varphi (\varphi -\varphi_\infty) - (\nabla J * (\varphi - \varphi_\infty))\varphi_\infty + \nabla \big(F(\varphi) - F(\varphi_\infty) + \frac{a}{2}(\varphi^2 -\varphi_\infty^2)\big).$$
Now by testing \eqref{eq4.7} with $\textbf{u}$ results,
\begin{align*}
    \nu\|\nabla \textbf{u}\|^2 + \eta \|\textbf{u}\|^2 &\leq C(\|\nabla a\|_{L^\infty} + \|\nabla J\|_{L^1}) \|\varphi -\varphi_\infty\|\|\textbf{u}\|.
\end{align*}
Applying Young's inequality and Korn's inequality, it follows that,
\begin{align}
    \|\textbf{u}(t)\|_{\mathbb{V}_{div}}^2 &\leq C\|\varphi(t) -\varphi_\infty\|^2
\end{align}
Therefore by \eqref{eq4.6}, we obtain the convergence $\textbf{u}(t) \rightarrow \textbf{0}$ in $\mathbb{V}_{div}$.
Hence, we conclude the convergence of the solution of the system \eqref{eq1}-\eqref{eq07} to a unique equilibrium point. We summarise the above discussion in the following result.

\begin{theorem}
Let all the assumptions in Theorem \ref{separation} hold. Additionally, assume $F$ is real analytic in $(-1,1)$. Then any weak solution of the system \eqref{eq1}-\eqref{eq07},  $(\varphi, \textbf{u})$ converges to an equilibrium point $(\varphi_\infty, \textbf{0})$.
    \begin{align*}
       \lim_{t\rightarrow \infty} \|\varphi(t)-\varphi_\infty\| + \|\textbf{u} (t)\|_{\mathbb{V}_{div}} = 0.
        \end{align*}     
\end{theorem}

\section{Conclusions}

 We have established the instantaneous separation property for the binary fluid evolving under CHB flow. This means that at any time the flow stays away from the pure phases and separation does not occur. Further, the $\delta$ which represents the measure of minimum separation; 
 %and from the estimate \eqref{d4}  it is evident that $\delta$  depends on $\tau$ and system parameters that includes $m, J, \Omega$ and the external forcing term in the velocity equation $\textbf{h}$ (see \eqref{d4}). }
 From \eqref{taucondn} it is evident that  the norm of $\textbf{h}$ in $L^2([0,\infty);\mathbb{V}_{div}')$ contributes to the separation scales only if it is  sufficiently large,  otherwise the coupling has no contribution in the separation. 
 
Our proof of separation property relies on the degeneracy of mobility at pure phases and a global estimates of the weak solution. We impose a growth condition on the mobility from above. An advantage of the method is that we can explicitly estimate the separation scale  from any time onwards. Up to our knowledge this is the first work using the De Giorgi iterative method to study the separation of a CH type system with degenerate type mobility and explicitly estimating the separation scale. Further the separation scale involves the forcing term in the velocity equation, which establishes the role of  Brinkman coupling  in separation of the phase variable. However, the methodology of the current work, cannot be directly used for the case $\alpha =1$ in \textbf{[A2]} which is the setting in most of the works on the singular potential and degenerate mobility.

Exploiting the separation property we derived improved regularity of the propagating solution. Further we discussed the convergence of solutions towards a unique stationary steady state.

%\section*{Declarations}
%\textbf{Ethical Approval:} Not applicable.
%\medskip\\
%\textbf{Funding:} The work of Greeshma K is supported by the Department of Science and Technology (DST), India, through the Innovation in Science Pursuit for Inspired Research (INSPIRE) Fellowship (IF210199). The work of Sheetal Dharmatti is supported by the Science and Engineering Research Board (SERB), Government of India, under grant SERB-CRG/2021/008278.
%\medskip\\
%\textbf{Clinical Trial Number:} Not applicable.
%\medskip\\
%\textbf{Conflicts of interest:} The authors declare that they have no conflicts of interest. 

%\printbibliography
%\nocite{*}

\bibliographystyle{plain}
\bibliography{ref}

@article{GGA,
  title={The separation property for 2D Cahn-Hilliard equations: Local, nonlocal and fractional energy cases},
  author={Gal, C. G. and Giorgini, A. and Grasselli, M.},
  journal={Discrete Contin. Dyn. Syst},
  volume={43},
  number={6},
  pages={2270--2304},
  year={2023}
}

@article{GGG,
   author = { Gal,C.G. and  Giorgini, A. and  Grasselli, M.},
   doi = {},
   issn = {},
   issue = {9},
   journal = {Journal of Differential Equations},
   month = {},
   pages = {5253-5297},
   title = {The nonlocal Cahn-Hilliard equation with singular potential: well-posedness, regularity and strict separation property},
   volume = {263},
   year = {2017},}

@article{POI,
  title={The 3D strict separation property for the nonlocal Cahn-Hilliard equation with singular potential},
  author={Poiatti, A. },
  journal={arXiv preprint arXiv:2303.07745},
  year={2023}
}

@article{GGF,
  title={Regularity results for the nonlocal Cahn-Hilliard equation with singular potential and degenerate mobility},
  author={Frigeri, S. and Gal, C. G. and Grasselli, M},
  journal={Journal of Differential Equations},
  volume={287},
  pages={295--328},
  year={2021},
  publisher={Elsevier}
}

@article{LSP,
  title={Convergence of solutions of a non-local phase-field system},
  author={Londen, S. O. and Petzeltov{\'a}, H.},
  journal={Discrete Contin. Dyn. Syst. Ser. S},
  volume={4},
  number={3},
  pages={653--670},
  year={2011}
}

@article{LPH,
  title={Regularity and separation from potential barriers for a non-local phase-field system},
  author={Londen, S. O. and Petzeltov{\'a}, H. },
  journal={Journal of Mathematical Analysis and Applications},
  volume={379},
  number={2},
  pages={724--735},
  year={2011},
  publisher={Elsevier}
}

@article{DPG,
   author = {Della Porta, F. and Grasselli, M.},
   doi = {},
   issn = {},
   issue = {},
   journal = {Communications on Pure and Applied Analysis},
   month = {},
   pages = {299-317},
   title = {On the non-local Cahn-Hilliard-Brinkman and Cahn-Hilliard-Hele-Shaw system},
   volume = {15},
   year = {2016}}

@article{SDM,
   author = {Dharmatti, S. and Mahendranath, P. L. N.},
   doi = {},
   issn = {},
   issue = {},
   journal = {Journal of Dynamical and Control Systems},
   month = {},
   pages = {221–246},
   title = {Non-local Cahn-Hiliard-Brinkman System with regular potential: regularity and optimal control},
   volume = {27},
   year = {2021},}

@article{SFG,
   author = { Frigeri,S. and  Gal, C.G. and Grasselli, M. and Sprekels, J.},
   doi = {},
   issn = {},
   issue = {},
   journal = {Nonlinearity},
   month = {},
   pages = {678-727},
   title = { Two- dimensional nonlocal Cahn– Hilliard/Navier– Stokes systems with variable viscosity, degenerate mobility and singular potential},
   volume = {32},
   year = {2019},}

@unpublished{SDK,
  title={Regularity and Optimal Control of Nonlocal Cahn Hilliard Brinkman System with a singular potential (Communicated)},
  author= {Dharmatti, S. and Greeshma, K.},
  year={2026},
  note = {Going to be appear in Zeitschrift für Analysis und ihre Anwendungen}}

@article{MCA,
  TITLE = {{The three-dimensional Cahn-Hilliard-Brinkman system with unmatched viscosities}},
  AUTHOR = {Conti, Monica and Giorgini, Andrea},
  JOURNAL = {https://hal.science/hal-01559179},
  YEAR = {2018},
}

@article{CFG,
  title={Non-local Cahn–Hilliard–Hele–Shaw systems with singular potential and degenerate mobility},
  author={ Cavaterra, C. and Frigeri, S. and  Grasselli, M.},
  journal={J. Math. Fluid Mech. },
  pages={},
  year={2022}
}

@book{TEM,
  title={Infinite dimensional dynamical systems in mechanics and physics},
  author={Temam, R.},
  volume={68},
  year={2012},
  publisher={Springer Science \& Business Media}
}

@article{AGS,
  title={On the separation property and the global attractor for the nonlocal Cahn-Hilliard equation in three dimensions},
  author={Giorgini, A.},
  journal={Journal of Evolution Equations},
  volume={24},
  number={2},
  pages={21},
  year={2024},
  publisher={Springer}
}

@article{DGG,
  title={The nonlocal Cahn--Hilliard--Hele--Shaw system with logarithmic potential},
  author={Della Porta, F.  and Giorgini, A.  and Grasselli, M. },
  journal={Nonlinearity},
  volume={31},
  number={10},
  pages={4851},
  year={2018},
  publisher={IOP Publishing}
}

@article{FRG,
year = {2015},
publisher = {IOP Publishing},
volume = {28},
number = {5},
pages = {1257},
author = {Frigeri, S. and Grasselli, M. and Rocca, E.},
title = {A diffuse interface model for two-phase incompressible flows with non-local interactions and non-constant mobility},
journal = {Nonlinearity},
}

@article {FRG26,
    AUTHOR = {Frigeri, S.},
     TITLE = {Uniqueness of weak solutions and strict separation property
              for the nonlocal {C}ahn-{H}illiard equation with singular
              potential, degenerate mobility, and sources},
   JOURNAL = {J. Evol. Equ.},
  FJOURNAL = {Journal of Evolution Equations},
    VOLUME = {26},
      YEAR = {2026},
    NUMBER = {1},
     PAGES = {Paper No. 26, 53},
}

@article {LUW25,
    AUTHOR = {Lv, M. and Wu, H.},
     TITLE = {On the {C}ahn-{H}illiard equation with kinetic rate dependent
              dynamic boundary condition and non-smooth potential:
              separation property and long-time behavior},
   JOURNAL = {J. Evol. Equ.},
  FJOURNAL = {Journal of Evolution Equations},
    VOLUME = {25},
      YEAR = {2025},
    NUMBER = {4},
     PAGES = {Paper No. 104, 38},
}

@article {LUWU,
    AUTHOR = {Lv, M. and Wu, H.},
     TITLE = {On the nonlocal {C}ahn-{H}illiard equation with nonlocal
              dynamic boundary condition and singular potential:
              well-posedness, regularity and asymptotic limits},
   JOURNAL = {Nonlinearity},
  FJOURNAL = {Nonlinearity},
    VOLUME = {38},
      YEAR = {2025},
    NUMBER = {4},
     PAGES = {Paper No. 045017, 62},
}

@article {LUH,
    AUTHOR = {Lv, M. and Wu, H.},
     TITLE = {Separation property and asymptotic behavior for a transmission
              problem of the bulk-surface coupled {C}ahn-{H}illiard system
              with singular potentials and its {R}obin approximation},
   JOURNAL = {Anal. Appl. (Singap.)},
  FJOURNAL = {Analysis and Applications},
    VOLUME = {23},
      YEAR = {2025},
    NUMBER = {6},
     PAGES = {887--928},
}

@article {STG,
    AUTHOR = {Stange, J.},
     TITLE = {Well-posedness and long-time behavior of a bulk-surface
              {C}ahn-{H}illiard model with non-degenerate mobility},
   JOURNAL = {Nonlinear Anal.},
  FJOURNAL = {Nonlinear Analysis. Theory, Methods \& Applications. An
              International Multidisciplinary Journal},
    VOLUME = {268},
      YEAR = {2026},
     PAGES = {Paper No. 114060, 34},
}

@article {BCG,
    AUTHOR = {Bosia, S. and Conti, M. and Grasselli, M.},
     TITLE = {On the {C}ahn-{H}illiard-{B}rinkman system},
   JOURNAL = {Commun. Math. Sci.},
  FJOURNAL = {Communications in Mathematical Sciences},
    VOLUME = {13},
      YEAR = {2015},
    NUMBER = {6},
     PAGES = {1541--1567},
}

@article {CGGG,
    AUTHOR = {Conti, M. and Galimberti, P. and Gatti, S. and
              Giorgini, A.},
     TITLE = {New results for the {C}ahn-{H}illiard equation with
              non-degenerate mobility: well-posedness and longtime behavior},
   JOURNAL = {Calc. Var. Partial Differential Equations},
  FJOURNAL = {Calculus of Variations and Partial Differential Equations},
    VOLUME = {64},
      YEAR = {2025},
}

@article{GAP,
  title={Convergence to equilibrium of weak solutions to the Cahn--Hilliard equation with non-degenerate mobility and singular potential},
  author={Grasselli, M. and Poiatti, A. },
  journal={arXiv preprint arXiv:2510.17296},
  year={2025}
}

\end{document}